\documentclass{article}

\usepackage{amssymb,amsmath,color}
\usepackage{amsthm}
\usepackage{url}
\newtheorem{theorem}{Theorem}[section]
\newtheorem{corollary}[theorem]{Corollary}

\newtheorem{conjecture}[theorem]{Conjecture}

\theoremstyle{definition}

\begin{document}

\title{On the Partitions into\\ Distinct Parts and Odd Parts}
\author{Mircea Merca\\ \\
	\footnotesize Department of Mathematics, University of Craiova, Craiova, DJ 200585, Romania\\
	\footnotesize Academy of Romanian Scientists, Ilfov 3, Sector 5, Bucharest, Romania\\
	\footnotesize mircea.merca@profinfo.edu.ro
%	\and \samethanks
%	\\ 
%	\footnotesize Department of Mathematics\\
%	\footnotesize University of Craiova\\
%	\footnotesize Craiova, DJ 200585, Romania\\
%	\footnotesize cpniculescu@gmail.com
}
\date{}
\maketitle
%\ccom{} \mcom{}

\begin{abstract}
	In this paper, we show that the difference between the number of parts in the odd partitions of $n$ 
	and the number of parts in the distinct partitions of $n$ satisfies Euler's recurrence relation for the partition
	function $p(n)$ when $n$ is odd. A decomposition of this difference in terms of the total number of parts in all the partitions of $n$ is also derived.
	In this context, we conjecture that for $k>0$, the series
	$$
	(q^2;q^2)_\infty \sum_{n=k}^\infty \frac{q^{{k\choose 2}+(k+1)n}}{(q;q)_n}
	\begin{bmatrix}
	n-1\\k-1
	\end{bmatrix}
	$$
	has non-negative coefficients.
\end{abstract}

\section{Introduction}

A partition of a positive integer $n$ is a sequence of positive integers whose sum is $n$. The order of the summands is unimportant when writing the partitions of $n$, but for consistency, a partition of $n$ will be written with the summands in a nonincreasing order \cite{Andrews76}. As usual, we denote by $p(n)$ the number of the partitions of $n$. For example, we have $p(5)=7$ because the partitions of $5$ are given as:
$$ 5,\ 4+1,\ 3+2,\ 3+1+1,\ 2+2+1,\ 2+1+1+1,\ 1+1+1+1+1.$$
The fastest algorithms for enumerating all the partitions of an integer have recently been presented by Merca \cite{Merca12,Merca13}.

One of the well-known theorems in the partition theory is Euler's pentagonal number theorem, i.e.,
$$\sum_{n=-\infty}^\infty (-1)^n q^{n(3n-1)/2} = (q;q)_\infty.$$
Here and throughout this paper, we use the following customary $q$-series notation:
$$(a;q)_0=1,\qquad (a;q)_n=\prod_{k=0}^{n-1}(1-aq^k),\qquad (a;q)_\infty=\prod_{k=0}^{\infty}(1-aq^k).$$
Because the infinite product $(a;q)_\infty$ diverges when $a\neq 0$ and 
$|q| \geqslant 1$, whenever $(a;q)_\infty$ appears in a formula, 
we shall assume that $|q| <1$.
Euler's pentagonal number theorem gives an easy linear recurrence relation for $p(n)$, namely
\begin{equation}\label{EPR}
\sum_{j=-\infty}^\infty (-1)^j p\big(n-j(3j-1)/2\big) = \delta_{0,n},
\end{equation}
where $\delta_{i,j}$ is the Kronecker delta function and $p(n)=0$ if $n<0$.

A famous theorem of Euler asserts that there are as many partitions of $n$ into distinct parts as there are partitions into odd parts \cite[p. 5. Cor. 1.2]{Andrews76}. For instance, the odd partitions of $5$ are:
 $$5,\quad 3+1+1\quad \text{and}\quad 1+1+1+1+1,$$ 
 while the distinct partitions of $5$ are: 
 $$5,\quad 4+1\quad \text{and}\quad 3+2.$$

We recall Euler's bijective proof of this result \cite{Glaisher}: A partition into distinct parts can be written as
$$n=d_1+d_2+\cdots+d_k.$$
Each integer $d_i$ can be uniquely expressed as a power of $2$ times an odd number, i.e.,
$$n=2^{\alpha_1} o_1 + 2^{\alpha_2} o_2 + \cdots + 2^{\alpha_k} o_k$$
where each $o_i$ is an odd number. Grouping together the odd numbers, we get the following expression
$$n=t_1\cdot 1 + t_3\cdot 3 + t_5\cdot 5 + \cdots,$$
where $t_i\geqslant 0$. 
If $d_i$ is odd, then we have $\alpha_i=0$. For $d_i$ even, it is clear that $\alpha_i>0$. So we deduce that
$$(t_1+t_3+t_5+\cdots)-k\geqslant 0,$$
for any positive integer $n$. In other words, the difference between the number of parts in the odd partitions of $n$ and the number of parts in the distinct partitions of $n$ is nonnegative. A combinatorial interpretation of this difference has been conjectured recently by George Beck \cite[A090867, Apr 22 2017]{Sloane}.

\begin{conjecture}
	The difference between the number of parts in the odd partitions of $n$ and the number of parts in the distinct partitions of $n$ equals the number of partitions of $n$ in which the set of even parts has only one element.
\end{conjecture}

A few days later, George E. Andrews \cite[Theorem 1]{Andrews17} provides a solution for this Beck's problem and introduces a new combinatorial interpretation for the difference between the number of parts in the odd partitions of $n$ and the number of parts in the distinct partitions of $n$.

\begin{theorem}\label{T1}
	For all $n\geqslant 1$, $a(n)=b(n)=c(n)$,
	where:
	\begin{enumerate}
		\item[-] $a(n)$ is the number of partitions of $n$ in which the set of even parts has only one element;
		\item[-] $b(n)$ is the difference between the number of parts in the odd partitions of $n$ and the number of parts in the distinct partitions of $n$;
		\item[-] $c(n)$ is the number of partitions of $n$ in which exactly one part is repeated.
	\end{enumerate}
\end{theorem}

For example, $a(5)=4$ because the four partitions in question are: 
$$4+1,\quad 3+2,\quad 2+2+1\quad \text{and} \quad 2+1+1+1.$$ 
We have already seen there are $9$ parts in the odd partitions of $5$ and $5$ parts in the distinct partitions of $5$ with the difference $b(5)=4$. On the other hand, we have $c(5)=4$ where the relevant partitions are: 
$$3+1+1,\quad 2+2+1,\quad 2+1+1+1\quad \text{and} \quad 1+1+1+1+1.$$

In this paper, inspired by Andrews's proof of Theorem \ref{T1},
we provide new properties for the difference between the number of parts in the odd partitions of $n$ and the number of parts in the distinct partitions of $n$ considering 
two factorizations for the generating function of $b(n)$.
%For what follows, we use the notation
%$$(a;q)_{\infty} = \prod_{n=0}^{\infty} (1-aq^n).$$
%Because this infinite product diverges when $a\neq 0$ and $|q| \geqslant 1$, whenever $(a;q)_\infty$ appears in a formula, we shall assume that $|q| <1$. 

This paper is organized as follows. In Section \ref{S2} we will show that the difference between the number of parts in the odd partitions of $n$ and the number of parts in the distinct partitions of $n$ satisfies Euler's recurrence relation \eqref{EPR} when $n$ is odd. In Section  \ref{S3} we will provide a decomposition of $b(n)$ in terms of the total number of parts in all the partitions of $n$.
%In Section \ref{S4} we present new connections between various partition functions and the divisor function $\tau(n)$ which counts the number of positive divisors of $n$.
A linear homogeneous inequality for the difference $b(n)$ are conjectured in Section \ref{S5} in analogy with the linear homogeneous inequality for Euler's partition function $p(n)$ provided by Andrews and Merca in \cite{Andrews12}.

\section{A pentagonal number recurrence for $b(n)$}
\label{S2}

In this section we consider $s(n)$ to be the difference between the number of parts in all the partitions of $n$ into odd number of distinct parts and the number of parts in all the partitions of $n$ into even number of distinct parts. For instance, considering the partitions of $5$ into distinct parts, we see that $$s(5)=1-2-2=-3.$$

In \cite{Andrews12}, Andrews and Merca 
defined $M_k(n)$ to be the number of partitions of $n$ in which $k$ is the least positive integer that is not a part and there are more parts $>k$ than there are parts $<k$. 
If $n=18$ and $k=3$ then we have $M_3(18)=3$ because the three partitions in question are: 
$$5+5+5+2+1,\quad 6+5+4+2+1,\quad \text{and} \quad 7+4+4+2+1.$$

We have the following result.

\begin{theorem}\label{TH2}
	Let $k$ and $n$ be positive integers. The partition functions $b(n)$, $s(n)$ and $M_k(n)$ are related by
	\begin{align*}
	& (-1)^{k-1} \left( \sum_{j=-(k-1)}^k (-1)^j b\big(n-j(3j-1)/2\big) -\frac{1+(-1)^n}{2} s\left( \frac{n}{2}\right) \right) \\
	&  = \sum_{j=1}^{\lfloor n/2 \rfloor} s(j) M_k(n-j).
	\end{align*}	
\end{theorem}

\begin{proof}
	As we can see in \cite{Andrews17}, the proof of Theorem \ref{T1} invokes the equality of the generating functions for $a(n)$, $b(n)$ and $c(n)$. So we consider the following factorization of Andrews for the generating function of $b(n)$:
	\begin{equation}\label{eq:GEAF}
	\sum_{n=0}^{\infty} b(n) q^n = (-q;q)_\infty \sum_{n=1}^{\infty} \frac{q^{2n}}{1-q^{2n}}.
	\end{equation}	
	%	where
	%$$(a;q)_\infty = \prod_{n=0}^{\infty} (1-aq^n).$$
	
	On the other hand, the identity
	\begin{equation*}
	\sum_{n=1}^\infty \frac{q^n}{1-q^n} = \frac{1}{(q;q)_\infty} \sum_{n=1}^{\infty} s(n) q^n.
	\end{equation*}
	is a specialization of the Lambert series factorization theorem \cite[Theorem 1.2]{Merca}.
	A proof of this relation via logarithmic differentiation can be seen in \cite[Theorem 1]{Merca16}.
	
	We have
	\begin{align}
	\sum_{n=0}^{\infty} b(n) q^n & = \frac{(q^2;q^2)_\infty}{(q;q)_\infty} \sum_{n=1}^{\infty} \frac{q^{2n}}{1-q^{2n}}\nonumber \\
	& = \frac{(q^2;q^2)_\infty}{(q;q)_\infty} \cdot \frac{1}{(q^2;q^2)_\infty} \sum_{n=1}^{\infty} s(n) q^{2n}\nonumber \\
	& = \frac{1}{(q;q)_\infty} \sum_{n=1}^{\infty} s(n) q^{2n}.\label{LS}
	\end{align}
	In \cite{Andrews12}, the authors considered Euler's pentagonal number theorem 
	and proved the following truncated form:
	\begin{equation} \label{TPNT}
	\frac{(-1)^{k-1}}{(q;q)_\infty} \sum_{n=-(k-1)}^{k} (-1)^{n} q^{n(3n-1)/2}= (-1)^{k-1}+ \sum_{n=k}^\infty \frac{q^{{k\choose 2}+(k+1)n}}{(q;q)_n}
	\begin{bmatrix}
	n-1\\k-1
	\end{bmatrix},
	\end{equation}
	where
	$$
	\begin{bmatrix}
	n\\k
	\end{bmatrix} 
	=
	\begin{cases}
	\dfrac{(q;q)_n}{(q;q)_k(q;q)_{n-k}}, &  \text{if $0\leqslant k\leqslant n$},\\
	0, &\text{otherwise}
	\end{cases}
	$$
	is the $q$-binomial coefficient.
	
	Multiplying both sides of \eqref{TPNT} by $\sum\limits_{n=1}^{\infty} s(n) q^{2n}$,
	we obtain
	\begin{align*}
	& (-1)^{k-1} \left( \bigg( \sum_{n=1}^\infty b(n) q^n \bigg) \bigg(  \sum_{n=-(k-1)}^{k} (-1)^{n} q^{n(3n-1)/2}\bigg) -\sum_{n=1}^\infty s(n)q^{2n}\right)   \\
	& =   \left( \sum_{n=1}^\infty s(n) q^{2n} \right) \left( \sum_{n=0}^\infty M_k(n) q^n\right),
	\end{align*}
	where we have invoked 
	the generating function for $M_k(n)$ \cite{Andrews12},
	$$
	\sum_{n=0}^\infty M_k(n) q^n = \sum_{n=k}^\infty \frac{q^{{k\choose 2}+(k+1)n}}{(q;q)_n}
	\begin{bmatrix}
	n-1\\k-1
	\end{bmatrix}.
	$$
	The proof follows easily considering Cauchy's multiplication of two power series.
\end{proof}

The limiting case $k\to\infty$ of Theorem \ref{TH2} provides the following linear recurrence relation for $b(n)$ involving the generalized pentagonal numbers.

\begin{corollary}\label{C3}
	For $n\geqslant 0$, 
	$$
	\sum_{k=-\infty}^{\infty} (-1)^k b\big(n-k(3k-1)/2\big) =
	\begin{cases}
	s(n/2),& \text{for $n$ even,}\\
	0,& \text{for $n$ odd.}
	\end{cases}
	$$
\end{corollary}

Theorem \ref{TH2} can be seen as a truncated form of Corollary \ref{C3}.
Considering again the relation \eqref{LS}, we remark the following convolution identity.

\begin{corollary}\label{C3a}
	For $n\geqslant 0$, 
	$$
	b(n) = \sum_{j=0}^{\lfloor n/2 \rfloor} s(j) p(n-2j).
	$$
\end{corollary}

%\begin{theorem}\label{T2}
%	For $|q|<1$,
%	\begin{equation*}\label{eq:3.1}
%	\sum_{n=0}^{\infty} b(n) q^n = \frac{1}{(q;q)_\infty}  \sum_{n=1}^{\infty} s(n) q^{2n},	
%	\end{equation*}
%	where $s(n)$ denotes the difference between the number of parts in all partitions of $n$ into odd number of distinct parts and the number of parts in all partitions of $n$ into even number of distinct parts.
%\end{theorem}

\section{A decomposition of $b(n)$}
\label{S3}

Let us define $S(n)$ to be  the total number of parts in all the partitions of $n$. For example, we have 
$$S(5)=1+2+2+3+3+4+5=20.$$

Andrews and Merca \cite{Andrews18}
defined $MP_k(n)$ to be the number of partitions of $n$ in which 
the first part larger than $2k-1$ 
is odd and appears exactly $k$ times. 
All other odd parts appear at most once.
For example, $MP_2(19) =10$, and the partitions in question are:
\begin{align*}
& 9 +9 +1,\ 9 +5 +5,\ 8 +5 +5 +1,\ 7 +7 +3 +2,\ 7 +7 +2 +2 +1,\\
& 7 +5 +5 +2,\ 6 +5 +5 +3,\ 6 +5 +5 +2 +1,\ 5 +5 +3 +2 +2 +2,\\
& 5 +5 +2 +2 +2 +2 +1.
\end{align*}

We have the following result.

\begin{theorem}\label{TH3}
	Let $k$ and $n$ be positive integers. The partition functions $b(n)$, $S(n)$ and $MP_k(n)$ are related by
	\begin{align*}
	b(n)-\sum_{j=0}^{2k-1} S\big(n/2-j(j+1)/4 \big) = (-1)^k \sum_{j=0}^n (-1)^j b(n-j) MP_k(j),
	\end{align*}
	where $S(x)=0$ if $x$ is not a positive integer
\end{theorem}

\begin{proof}
	First we want the generating function for
	partitions where $z$ keeps track of the number of parts equal to $k$. This is
	\begin{align*}
	\frac{1}{1-zq^k} \prod_{\substack{n=1\\n\neq k}}^\infty \frac{1}{1-q^n}
	= \frac{1}{(q;q)_\infty} \cdot \frac{1-q^k}{1-zq^k}.
	\end{align*} 
	Let $S_k(n)$ denote the total number of $k$'s in all the partitions of $n$. Hence
	\begin{align*}
	\sum_{n=0}^\infty S_k(n) q^n =  \frac{d}{dz}\Bigr|_{z=1} \frac{(1-q^k)}{(q;q)_\infty (1-zq^k)}=
	\frac{q^k}{1-q^k}\cdot \frac{1}{(q;q)_\infty}.
	\end{align*} 
	Thus, we deduce the following generating function for $S(n)$:
	$$\sum_{n=0}^\infty S(n) q^{n} = \frac{1}{(q;q)_\infty} \sum_{n=1}^\infty \frac{q^n}{1-q^n}.$$
	So we can write
	\begin{align*}
	\sum_{n=0}^{\infty} b(n) q^n & = \frac{(q^2;q^2)_\infty}{(q;q)_\infty} \sum_{n=1}^{\infty} \frac{q^{2n}}{1-q^{2n}}\\
	& = \frac{(q^2;q^2)_\infty}{(q;q^2)_\infty (q^2;q^2)_\infty} \sum_{n=1}^{\infty} \frac{q^{2n}}{1-q^{2n}}\\
	& = \frac{(q^2;q^2)_\infty}{(q;q^2)_\infty} \sum_{n=0}^\infty S(n) q^{2n}.
	\end{align*}
	This identity can be written as follows:
	\begin{align}
	& \frac{(q;q^2)_\infty}{(q^2;q^2)_\infty}
	\sum_{n=0}^{\infty} b(n) q^n = \sum_{n=0}^{\infty} S(n) q^{2n}. \label{LS1}
	\end{align}
	
	In \cite{Andrews18}, the authors considered the following theta identity of Gauss
	\begin{equation}\label{Gauss}
	\sum_{n=0}^{\infty} (-q)^{n(n+1)/2} = \frac {(q^2;q^2)_\infty} {(-q;q^2)_\infty}
	\end{equation}
	and proved the following truncated form:
	\begin{align*}
	& \frac{(-q;q^2)_{\infty}} {(q^2;q^2)_{\infty}} \sum_{j=0}^{2k-1} (-q)^{j(j+1)/2}  \\
	& \qquad = 1+(-1)^{k-1} \frac{(-q;q^2)_k}{(q^2;q^2)_{k-1}} \sum_{j=0}^{\infty} \frac{q^{k(2k+2j+1)}(-q^{2k+2j+3};q^2)_{\infty}}{(q^{2k+2j+2};q^2)_{\infty}}.
	\end{align*}
	By this relation, with $q$ replaced by $-q$, we obtain
	\begin{align}
& \frac{(q;q^2)_{\infty}} {(q^2;q^2)_{\infty}} \sum_{j=0}^{2k-1} q^{j(j+1)/2} 
 = 1+(-1)^{k-1} \sum_{n=0}^\infty (-1)^n MP_k(n)q^n,\label{eq4} 
\end{align}	
	where we have invoked 
the generating function for $MP_k(n)$ \cite{Andrews18},
$$
\sum_{n=0}^\infty MP_k(n) q^n = 
\frac{(-q;q^2)_k}{(q^2;q^2)_{k-1}} \sum_{j=0}^{\infty} \frac{q^{k(2k+2j+1)}(-q^{2k+2j+3};q^2)_{\infty}}{(q^{2k+2j+2};q^2)_{\infty}}.
$$
	
	Multiplying both sides of \eqref{eq4} by $\sum\limits_{n=0}^\infty b(n) q^{n}$,
	we obtain
	\begin{align*}
	& (-1)^{k-1} \left( \bigg( \sum_{n=1}^\infty S(n) q^{2n} \bigg) \bigg( \sum_{n=0}^{2k-1} q^{n(n+1)/2}\bigg) -\sum_{n=1}^\infty b(n)q^{n}\right)   \\
	& =   \left( \sum_{n=1}^\infty b(n) q^{n} \right) \left( \sum_{n=0}^\infty (-1)^n MP_k(n) q^n\right).
	\end{align*}

	The proof follows easily considering Cauchy's multiplication of two power series.
	
\end{proof}
%\begin{theorem}\label{T3}
%	For $|q|<1$,
%	\begin{equation*}\label{eq:3.1a}
%	\sum_{n=0}^{\infty} b(n) q^n = \frac{(q^2;q^2)_\infty}{(q;q^2)_\infty}  \sum_{n=1}^{\infty} S(n) q^{2n},	
%	\end{equation*}
%	where $S(n)$ denotes the total number of parts in all partitions of $n$ .
%\end{theorem}

%We shall remark the following consequence of Theorem \ref{T2}.

The limiting case $k\to\infty$ of Theorem \ref{TH3} provides the following decomposition of the difference $b(n)$ in terms of $S(n)$.

\begin{corollary}\label{C4}
	For $n\geqslant 0$,
	\begin{equation*}
b(n) = \sum_{k=0}^{\infty} S\big(n/2-k(k+1)/4 \big),	
\end{equation*}	
	with $S(x)=0$ if $x$ is not a positive integer.
\end{corollary}

More explicitly, this corollary can be rewritten as:
$$
b(2n)=\sum_{k=-\infty}^\infty S\big(n-k(4k-1)\big)
$$
and
$$
b(2n+1)=\sum_{k=-\infty}^\infty S\big(n-k(4k-3)\big).
$$
Combinatorial proofs of these identities would be very interesting. 
On the other hand, the relation \eqref{LS1} allows us to remark that
$$S(n)=\sum_{k=0}^{2n} (-1)^k e(k) b(2n-k),$$
where $e(n)$ is the number of partitions of $n$ in which each even part occurs with even multiplicity and there is no restriction on the odd parts \cite[A006950]{Sloane}.
Other properties for $S(n)$ can be found in \cite{Knopfmacher}.

As a consequence of Theorem \ref{TH3}, we remark the following infinite families of inequalities involving the partition functions $b(n)$ and $MP_k(n)$.

\begin{corollary}
	Let $k$ and $n$ be positive integers. Then
	\begin{align*}
 (-1)^k \sum_{j=0}^n (-1)^j b(n-j) MP_k(j) \geqslant 0.
	\end{align*}
\end{corollary}

\begin{proof}
	We take into account that
	\begin{align*}
	 b(n)-\sum_{j=0}^{1} S\big(n/2-j(j+1)/4 \big) 
	& \geqslant b(n)-\sum_{j=0}^{3} S\big(n/2-j(j+1)/4 \big)\\
	& \geqslant \cdots \geqslant b(n)-\sum_{j=0}^{\infty} S\big(n/2-j(j+1)/4 \big) = 0.
	\end{align*}
\end{proof}

Relevant to Theorem \ref{TH3}, it would be very appealing to have combinatorial
interpretations of
$$(-1)^k \sum_{j=0}^n (-1)^j b(n-j) MP_k(j).$$

\section{Open problems}
\label{S5}

Linear homogeneous inequalities involving Euler's partition function $p(n)$ have been the subject
of recent studies \cite{Andrews12,Andrews18,Merca12,Merca19}. In \cite{Merca12}, the author proved the inequality
\begin{equation*}\label{eq1}
p(n)-p(n-1)-p(n-2)+p(n-5)\leqslant 0,\qquad n>0,
\end{equation*}
in order to provide the fastest known algorithm for the generation of the partitions of $n$.
Subsequently, Andrews and Merca \cite{Andrews12} proved more generally that, for $k>0$,
\begin{equation}\label{ineq2}
(-1)^{k-1} \sum_{j=-(k-1)}^{k} (-1)^{j} p\big(n-j(3j-1)/2\big) \geqslant 0,
\end{equation} 
with strict inequality if $n\geqslant k(3k+1)/2$. In other words, for $k>0$, the coefficients of $q^n$ in the series
$$
(-1)^{k-1} \left( \frac{1}{(q;q)_\infty} \sum_{j=-(k-1)}^{k} (-1)^{j} q^{j(3j-1)/2}-1\right) 
$$
are all zero for $0\leqslant n < k(3k+1)/2$, and for  $n\geqslant k(3k+1)/2$ all the coefficients are positive.
Related to this result on truncated pentagonal number series, we remark that
there is a substantial amount of numerical evidence to conjecture a stronger result.

\begin{conjecture}\label{C41}
	For $k>0$, the coefficients of $q^n$ in the series
	$$
	(-1)^{k-1}  \left( \frac{1}{(q;q)_\infty} \sum_{j=-(k-1)}^{k} (-1)^{j} q^{j(3j-1)/2}-1\right) (q^2;q^2)_\infty
	$$
	are all zero for $0\leqslant n < k(3k+1)/2$, and for  $n\geqslant k(3k+1)/2$ all the coefficients are positive. 
\end{conjecture}

Let $Q(n)$ be the number of partitions of $n$ into odd parts. 
It is well known that the generating function for $Q(n)$ is $1/(q;q^2)_\infty$. 
Assuming Conjecture \ref{C41}, we immediately deduce that the partition functions $p(n)$ and $Q(n)$, share a common infinite family of linear inequalities of the form \eqref{ineq2} when $n$ is odd. 
In addition, considering Theorem \ref{TH2}, we easily deduce that the partition function $b(n)$ satisfies the following infinite families of linear inequalities.

\begin{conjecture}
	For $k>0$,
	$$
(-1)^{k-1} \left( \sum_{j=-(k-1)}^k (-1)^j b\big(n-j(3j-1)/2\big) -\frac{1+(-1)^n}{2} s\left( \frac{n}{2}\right) \right) \geqslant 0,
	$$
	with strict inequalities if  $n\geqslant2+k(3k+1)/2$.
\end{conjecture}

In this context, relevant to Theorem \ref{TH2}, it would be very appealing to have combinatorial
interpretations of
$$\sum_{j=1}^{\lfloor n/2 \rfloor} s(j) M_k(n-j).$$

\section{Concluding remarks}

New properties for the difference between the number of parts in the odd partitions of $n$ 
and the number of parts in the distinct partitions of $n$ have been introduced in this paper.

Surprisingly, when $n$ is odd, Euler's partition function $p(n)$ and the difference $b(n)$ share two common linear homogeneous recurrence relations. As we can see in Corollary \ref{C3}, the first recurrence relation involves the generalized pentagonal numbers:
\begin{align*}
p(n) &= p(n-1) + p(n-2) - p(n-5) -p(n-7) \\
& \qquad + p(n-12) + p(n-15) - p(n-22) - p(n-26) + \cdots,
\end{align*}
and
\begin{align*}
b(n) &= b(n-1) + b(n-2) - b(n-5) -b(n-7) \\
& \qquad + b(n-12) + b(n-15) - b(n-22) - b(n-26) + \cdots.
\end{align*}
The second recurrence relation combines the partition function $p(n)$ and the difference $b(n)$ with the triangular numbers, as follows:
\begin{align*}
p(n) &= p(n-1) + p(n-3) - p(n-6) -p(n-10) \\
& \qquad + p(n-15) + p(n-21) - p(n-28) - p(n-36) + \cdots,
\end{align*}
and
\begin{align*}
b(n) &= b(n-1) + b(n-3) - b(n-6) -b(n-10) \\
& \qquad + b(n-15) + b(n-21) - b(n-28) - b(n-36) + \cdots.
\end{align*} 
These relations can be easily derived considering again the theta identity of Gauss \eqref{Gauss} and the following two identities:
\begin{align*}
&\frac{(q^2;q^2)_\infty}{(-q;q^2)_\infty} \sum_{n=0}^\infty p(n) q^n = (-q^2;q^2)_\infty,
\end{align*}
and
\begin{align*}
&\frac{(q^2;q^2)_\infty}{(-q;q^2)_\infty} \sum_{n=0}^\infty b(n) q^n = (q^4;q^4)_\infty \sum_{n=1}^\infty \frac{q^{2n}}{1-q^{2n}}.
\end{align*}

Finally, we want to thank Professor George E. Andrews for his valuable comments on the first version of this paper.

%It is clear that the identity \eqref{Eq6} can be applied on Corollary \ref{C6}. Other identities involving coefficients binomials can be obtained in this way. 



\bigskip

%\noindent\textit{Department of Mathematics,
%University of Craiova, 200585 Craiova, Romania\\
%mircea.merca@profinfo.edu.ro}

\end{document}